\documentclass[11pt,english,a4paper]{amsart} 
\usepackage[T1]{fontenc}
\usepackage[latin1]{inputenc}
\usepackage[a4paper]{geometry}
\usepackage{mathrsfs}
\usepackage{amsthm}
\usepackage{amsmath}
\usepackage{amssymb}
\usepackage{graphicx}
\usepackage{caption}
\usepackage{subcaption}
\usepackage{color}
\usepackage{a4wide}
\usepackage{amscd}
\usepackage{latexsym}
\usepackage{dsfont}
\usepackage{soul}
\usepackage{float, subcaption, caption, wrapfig}

\makeatletter
  \theoremstyle{remark}
  \newtheorem*{acknowledgement*}{\protect\acknowledgementname}
  
  \theoremstyle{plain}
  \newtheorem{thm}{\protect\theoremname}[section]

  \theoremstyle{plain}
  \newtheorem{example}[thm]{\protect\examplename}
  
  \theoremstyle{plain}
  \newtheorem{defn}[thm]{\protect\definitionname}
  
  \theoremstyle{plain}
  \newtheorem{rem}[thm]{\protect\remarkname}
  
  \theoremstyle{plain}
  \newtheorem{problem}[thm]{\protect Problem}
  
  \theoremstyle{plain}
  \newtheorem{lem}[thm]{\protect\lemmaname}
    
  \theoremstyle{plain}
  \newtheorem{prop}[thm]{\protect Proposition}
  
  \theoremstyle{plain}

\usepackage{MnSymbol}
\usepackage{fancyhdr}
\usepackage{esint}

\makeatother

\usepackage{babel}
  \providecommand{\acknowledgementname}{Acknowledgement}
  \providecommand{\definitionname}{Definition}
  \providecommand{\examplename}{Example}
  \providecommand{\lemmaname}{Lemma}
  \providecommand{\remarkname}{Remark}
\providecommand{\theoremname}{Theorem}

\newcounter{todocounter}
\newlength{\todowidthinner}
\definecolor{grey}{rgb}{0.7,0.7,0.7}

\definecolor{rkcol}{rgb}{0,0.3,0.3}

\definecolor{mhcol}{rgb}{0,0.7,0}

\definecolor{jpcol}{rgb}{0.7,0.7,0}

\newcommand{\MOD}[1]{\textcolor{black}{#1}}

\begin{document}
\newcommand{\Sphere}{{\mathbb S}^d}

\global\long\def\Cc{\mathbb{C}}
\global\long\def\Ic{\mathbb{I}}

\global\long\def\R{\mathbb{R}}

\global\long\def\Rd{\mathbb{R}^{d}}
\global\long\def\Zd{\mathbb{Z}^{d}}

\global\long\def\N{\mathbb{N}}

\global\long\def\Q{\mathbb{Q}}

\global\long\def\Zz{\mathbb{Z}}

\global\long\def\eps{\varepsilon}

\global\long\def\cP{\mathcal{P}}
\global\long\def\cR{\mathcal{R}}
\global\long\def\cS{\mathcal{S}}
\global\long\def\cT{\mathcal{T}}
\global\long\def\cN{\mathcal{N}}

\global\long\def\sF{\mathscr{F}}

\global\long\def\ue{u^{\eps}}

\global\long\def\bQ{\boldsymbol{Q}}
\global\long\def\bomega{\boldsymbol{\omega}}
\global\long\def\cS{\mathcal{S}}
\global\long\def\BD{\mathrm{BD}}

\global\long\def\borelB{\mathcal{B}}

\global\long\def\baireM{\mathcal{M}}

\global\long\def\metricd{\mathsf{d}}

\global\long\def\DStar{\mathcal{D}^{*}}

\global\long\def\mapA{\mathcal{A}}

\global\long\def\lebesgueL{\mathcal{L}}

\global\long\def\mathS{\mathcal{S}}

\global\long\def\mathP{\mathcal{P}}

\global\long\def\muomega{\mu_{\omega}}

\global\long\def\muomegaeps{\mu_{\omega}^{\eps}}

\global\long\def\mutildeomegaeps{\mu_{\widetilde{\omega}}^{\eps}}

\global\long\def\mupalm{\mu_{\mathcal{P}}}

\global\long\def\nupalm{\nu_{\mathcal{P}}}

\global\long\def\mugammaomegaeps{\mu_{\Gamma(\omega)}^{\eps}}

\global\long\def\mugammaomega{\mu_{\Gamma(\omega)}}

\global\long\def\mugammapalm{\mu_{\Gamma,\mathcal{P}}}

\global\long\def\tauxeps{\tau_{\frac{x}{\eps}}}

\global\long\def\tildeomega{\widetilde{\omega}}

\global\long\def\wOmega{\widetilde{\Omega}}

\global\long\def\womega{\widetilde{\omega}}

\global\long\def\wsigma{\widetilde{\sigma}}

\global\long\def\wGamma{\widetilde{\Gamma}}

\global\long\def\wtau{\widetilde{\tau}}

\global\long\def\wdelta{\widetilde{\delta}}

\global\long\def\wphi{\widetilde{\phi}}

\global\long\def\wmu{\widetilde{\mu}}

\global\long\def\weta{\widetilde{\eta}}

\global\long\def\closedsets{\mathcal{F}}

\global\long\def\ttopology{\mathcal{T}_{\closedsets}}

\global\long\def\intrn{\int_{\R^{n}}}

\global\long\def\intomega{\int_{\Omega}}

\global\long\def\spaceomega{(\Omega,\borelB(\Omega),\mu)}

\global\long\def\sigmafinite{$\sigma$-finite~}

\global\long\def\limarrow#1#2{{\displaystyle \stackrel{\longrightarrow}{#1\rightarrow#2}}}

\global\long\def\Lpomega{L^{p}\spaceomega}

\global\long\def\nablax{\nabla_{x}}

\global\long\def\nablaomega{\nabla_{\omega}}

\renewcommand{\div}{{\text{div}\,}}

\global\long\def\divx{\textnormal{div}_{x}}

\global\long\def\divomega{\textnormal{div}_{\omega}}

\global\long\def\tsq#1{\stackrel{2s}{\rightharpoonup}_{#1}}

\global\long\def\weakto{\rightharpoonup}

\global\long\def\mikelic{Mikeli\'c}

\global\long\def\cnull{C_{0}(\Rr^{n})}

\global\long\def\tautildeomegaeps{\widetilde{\tau}_{\frac{x}\cE{\eps}\omega}}
\global\long\def\jump#1{\lsem#1\rsem}
\global\long\def\pmuj#1{\rsem#1\lsem}

\global\long\def\aom#1{a_{\omega,#1}}

\global\long\def\oaom#1{\overline{a}_{\omega,#1}}

\global\long\def\cC{\mathcal{C}}
\global\long\def\cL{\mathcal{L}}
\global\long\def\cE{\mathcal{E}}
\global\long\def\cF{\mathcal{F}}
\global\long\def\cG{\mathcal{G}}
\global\long\def\cH{\mathcal{H}}

\global\long\def\cM{\mathcal{M}}

\global\long\def\pot{\mathrm{pot}}
\global\long\def\sol{\mathrm{sol}}
\global\long\def\per{\mathrm{per}}
\global\long\def\lip{\mathrm{Lip}}

\global\long\def\diver{\mathrm{div}}
\global\long\def\d{\mathrm{d}}
\global\long\def\w{\mathrm{w}}
\global\long\def\Om{\mathrm{Om}}
\global\long\def\scrV{\mathscr{V}}
\global\long\def\scrW{\mathscr{W}}

\global\long\def\rmD{\mathrm{D}}

\global\long\def\boldeta{\boldsymbol{\eta}}

\global\long\def\fc{\mathfrak{c}}
\global\long\def\fa{\mathfrak{a}}
\global\long\def\fg{\mathfrak{g}}
\global\long\def\fA{\mathfrak{A}}
\global\long\def\fD{\mathfrak{D}}
\global\long\def\bn{\boldsymbol{n}}

\title{Fractal homogenization of multiscale interface problems}

\author{Martin Heida, Ralf Kornhuber, and Joscha Podlesny}
%
%

\date{}

\maketitle
\begin{abstract}
Inspired by  \MOD{continuum mechanical contact problems with} geological fault networks, 
we consider  elliptic second order differential equations with jump conditions 
on a sequence of multiscale networks of interfaces 
\MOD{with a finite number of non-separating scales.} 
\MOD{Our aim is to derive and analyze a description of 
the asymptotic limit of infinitely many scales 
in order to quantify the effect of resolving the network 
only up to some finite number of interfaces 
and to consider all further effects as homogeneous.}
As classical homogenization techniques are not suited 
for this kind of geometrical setting,  
we suggest a new concept, called fractal homogenization,
to \MOD{derive and analyze an asymptotic limit problem
from a corresponding sequence of finite-scale interface problems.}
We provide an intuitive characterization of 
the corresponding fractal solution space in terms of generalized jumps and gradients
together with continuous embeddings into $L^2$ and $H^s$, $s<1/2$. 
We  show existence and uniqueness of the solution of the asymptotic \MOD{limit} problem
and exponential convergence of the approximating 
\MOD{finite-scale} solutions.
Computational experiments involving a related numerical homogenization technique
illustrate  our theoretical findings.
\end{abstract}

\thanks{
This research has been funded by Deutsche Forschungsgemeinschaft (DFG)
through grant CRC 1114 ''Scaling Cascades in Complex Systems'', 
Project C05 
''Effective models for interfaces with many scales'' and 
Project B01 
''Fault networks and scaling properties of deformation accumulation''.
}

\section{Introduction}
Classical elliptic homogenization is concerned with  second order differential equations of the form
\begin{equation}\label{eq:abstract-bulk}
-\nabla\left(A^{\eps}\nabla u_{\eps}\right)=f\,,
\end{equation}
denoting $A^{\eps}(x)=A\left(\frac{x}{\eps}\right)$  with $\eps >0$ and some uniformly bounded, positive  coefficient field  $A$. 
Hence, $A^{\eps}$ is oscillating on a spatial scale of size $\eps$ compared to the diameter
of the macroscopic computational domain $\bQ\subset \Rd$. In periodic homogenization,
the coefficient $A$ is $\mathrm{Y}$-periodic, where $\mathrm{Y}=[0,1[^{d}$
is the unit cell in $\Rd$. In stochastic homogenization, 
the coefficient $A^{\eps}(x)=A_\omega\left(\frac x\eps\right)$ 
is a stationary (i.e. statistically shift invariant) 
and ergodic (asymptotically uncorrelated) random variable 
on a probability space $(\Omega, \mathcal{F}, P)$ with $\omega\in\Omega$. 
A variety of results have been derived in the field of  homogenization, and we refer to \cite{allaire1992homogenization,allaire1996multiscale,cioranescu2012periodic,hornung2012homogenization}
for the periodic case and to \cite{JKO1994,zhikov2006homogenization} for the stochastic case. 
For error estimates in homogenization, we refer to \cite{armstrong2017quantitative,armstrong2014quantitative,cioranescu2012periodic,gloria2014optimal,griso2004error}. 
%
Mathematical modelling of polycrystals or composite materials typically leads to elliptic interface problems
with appropriate jump conditions on a microscopic interface $\Gamma^\eps\subset \bQ$.
A periodic setting is obtained by  $\Gamma^\eps =\eps\Gamma_0$ with scaling parameter $\eps >0$  
and a piecewise smooth   hypermanifold $\Gamma_0$ with $\mathrm Y$-periodic cells.
The size of the cells is then of order $\eps$  compared to the macroscopic domain.
Denoting by  $\jump{u_\eps}_\nu$ the jump of $u_\eps$ in normal direction $\nu$ on $\Gamma^\eps$,
the condition
\begin{equation}\label{eq:abstract-Gamma}
-\partial_\nu u_\eps=\jump{u_\eps}_\nu
\end{equation}
on the normal derivatives  $\partial_\nu u_\eps$  is imposed at the boundary of each cell. 
Corresponding stochastic variants have been studied in \cite{heida2011extension,Hummel1999}. 
The homogenization of such kind of periodic multiscale interface problems
has been studied in great detail, see \cite{cioranescu2013homogenization,donato2004homogenization,gruais2017heat,heida2012stochastic} and references therein. 
\MOD{Similar concepts have been applied to foam-like 
elastic media like the human lung, cf., e.g., \cite{baffico2008homogenization,cazeaux2015homogenization}.
}
%
Classical (stochastic) homogenization relies on periodicity (or ergodicity) and scale separation.
The latter means that  homogenized problems in the asymptotic limit $\eps \to 0$ 
usually decouple into a global problem
that describes the macroscopically observed behavior of the system,
and one or more local problems, often referred to as \emph{cell-problems},
that capture the oscillatory behavior. 

In contrast to analytic homogenization,
numerical homogenization addresses the lack of regularity of solutions of problems with highly oscillatory coefficients $A^{\eps}$
in numerical computations,
either by local corrections of standard finite elements~\cite{efendiev2009multiscale,MalqvistPeterseim14} or 
by multigrid-type iterative schemes~\cite{KornhuberPodlesnyYserentant17,KornhuberYserentant16}.
Both approaches are closely related~\cite{KornhuberPeterseimYserentant16} 
and usually do not rely on periodicity or scale separation.

In this work, we consider elliptic multiscale interface problems 
without scale separation in a non-periodic geometric setting motivated by geology. 
Experimental studies suggest that  grains in fractured rock are 
distributed in a fractal manner~\cite{nagahama1994scaling,turcotte1994crustal}.
In particular, this means that the size of grains and interfaces follows an exponential law:
The  total number $N(r)$ of  grains larger than some $r>0$ behaves according to 
\begin{equation}\label{eq:fractal-distribution}
N(r)=Cr^{-D}\,
\end{equation}
and $D$ is often called the fractal dimension. 
This observation is also captured by  geophysical modelling of fragmentation by tectonic deformation~\cite{sammis1986self}
which is based on the assumption that deformation of two neighboring blocks of equal
size  might lead to fracturing in
one of these blocks.
It is  unlikely and therefore excluded in this model, that  bigger blocks break smaller ones or vice versa. 
A typical example for corresponding multiscale interface networks 
is given by the Cantor-type  geometry~\cite{turcotte1994crustal}
as depicted in Figure~\ref{fig:Cantor-Sets}. 
\MOD{
While each level-$K$ interface network $\Gamma^{(K)}$ clearly is two-dimensional, 
the limiting multiscale network $\Gamma=\Gamma^{(\infty)}$
} 
has fractal dimension $\ln6/\ln2$, which is in good
agreement with experimental studies that often yield $D\approx2.5$.
Observe that the cells representing the different grains are not periodically distributed. 
They can also be arbitrarily small 
and cover the whole range up to half of the given domain $\bQ$
so that there is no scale parameter $\eps$ separating a small from a large scale.
\MOD{Similar geometric settings, but with a completely different scope, 
occur for thin fractal fibers~\cite{mosco2013thin}.
}

\begin{figure}
\centering
\def\width{4cm}
\includegraphics[width=\width]{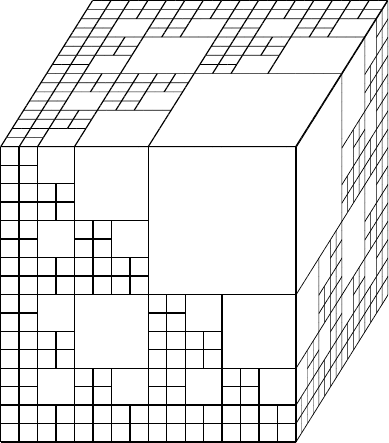}
\includegraphics[width=\width]{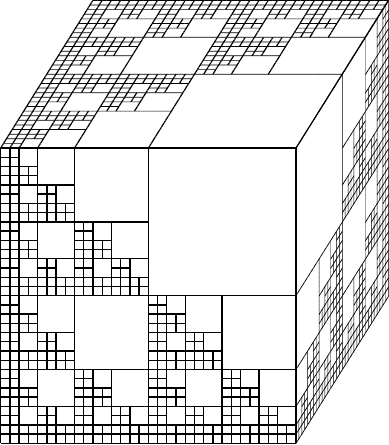}
\includegraphics[width=\width]{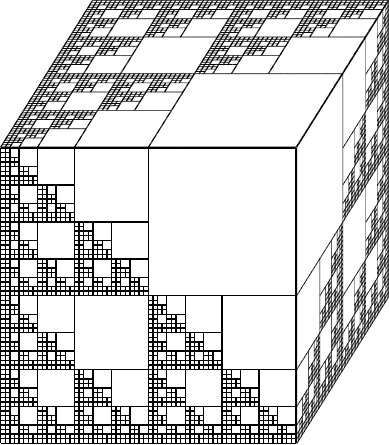}
\caption{\label{fig:Cantor-Sets}
Level-$K$ interface network $\Gamma^{(K)}$ for $K=4$, $5$, and $6$,
taken from~\cite{turcotte1994crustal}.}
\end{figure} 

\MOD{Geological applications give rise to 
continuum mechanical problems with frictional contact 
on such multiscale networks of interfaces or faults.
The level-$K$ network $\Gamma^{(K)}= \bigcup_{k=1}^K \Gamma_k$ 
consists of single faults $\Gamma_k$ which are ordered from strong to weak
in the sense that discontinuities of displacements along $\Gamma_k$ are expected 
to decrease for increasing $k$,
because ``more fractured'' media are expected to show higher resistance
(for a more detailed dicussion, see, e.g., \cite{ben2003characterization,gao2014strength,oncken2012strain} 
and the references cited therein).
}

\MOD{
In this paper, we restrict our considerations to  
scalar elliptic model problems on $\bQ\setminus\Gamma^{(K)}$ for each level $K\in \N$ 
with weighted jumps along the network of interfaces $\Gamma^{(K)}$, 
 instead of nonlinear frictional contact conditions. 
The ordering of the single interfaces
$\Gamma_k$ from strong to weak is reflected by scaling the contributions from the 
jumps along $\Gamma_k$ in the corresponding energy functional 
with exponential weights $C_k(1+\fc)^k$.
Here, $C_k>0$ is a geometrical constant measuring the rate of fracturing for each $k$
and $\fc>0$ is a kind of material constant that determines 
the growth of resistance to jumps with increasing fracturing.
We exploit the hierarchical structure of the interface networks $\Gamma^{(K)}$
to derive a hierarchy of solution spaces $\cH_K$
for the above-mentioned level-$K$ interface problems. 
Under usual ellipticity conditions, the problems
admit unique solutions $u_K \in \cH_K$ for all $K\in \N$.
The main concern of this paper is to investigate the 
asymptotic behavior of $u_K$ for $K\to \infty$.
As classical  homogenization techniques are not suited for this purpose,
we develop a new concept called \emph{fractal homogenization}.
The starting point is the construction of an asymptotic fractal limit space $\cH$, 
that arises in a natural way
by completion of the union of the level-$K$ spaces $\cH_K$, $K\in \N$.
We provide continuous embeddings
$\cH\subset L^2$ and $\cH\subset H^s$, $s<\frac{1}{2}$,
and a characterization of $\cH$ in terms of generalized jumps and gradients.
We then formulate a  fractal limit problem associated with the level-$K$ interface problems
and show existence of a unique solution $u \in \cH$ 
together with convergence $u_K \to u$ in $\cH$.
Imposing additional regularity assumptions on the geometry 
of the  multiscale  interface networks $\Gamma^{(K)}$, $K \in \N$, 
we are able to even show exponential estimates 
of the fractal homogenization error $\|u-u_K\|$ in $\cH$ for $K\to \infty$.
}
In order to illustrate our theoretical findings by numerical experiments, 
we introduce a fractal numerical homogenization scheme in the spirit 
of~\cite{KornhuberPodlesnyYserentant17,KornhuberYserentant16} 
that is based on  a hierarchy of local patches
from a hierarchy of meshes ${\mathcal T}_1$, ..., ${\mathcal T}_K$ 
successively resolving the interfaces $\Gamma^{(1)}$, ..., $\Gamma^{(K)}$.
This decomposition induces an additive Schwarz preconditioner to accelerate
the convergence of a conjugate gradient iteration. 
In numerical experiments with a Cantor-type geometry,
we found the theoretically predicted behavior of (finite element approximations $\tilde{u}_K$ of) $u_K$. 
We also observed that the convergence rates of our  iterative scheme appear to be robust with respect to increasing $K$. 
Theoretical justification and extensions to model reduction 
in the spirit of~\cite{KornhuberPeterseimYserentant16,MalqvistPeterseim14} 
are subject of current research. 

\MOD{
The paper is structured as follows. 
In Section~\ref{sec:MultiscaleInterface},
we introduce multiscale interface networks
together with associated level-$K$ interface problems
and prove existence and uniqueness of solutions $u_K$, $K\in \N$. 
In Section~\ref{sec:Main-results}, we derive and analyze an
associated fractal limit space $\cH$ and provide some basic properties, 
such as Sobolev embeddings and a Poincar\'e-type inequality.
Then, we introduce a fractal interface problem, show
existence of a unique solution $u\in \cH$ 
as well as convergence $u_K \to u$ in $\cH$.
Exploiting additional assumptions on the geometry,
we prove exponential homogenization error estimates in 
Section~\ref{sec:error-estimates}.
Section~\ref{sec:numerics}  is devoted to 
numerical computations based on (fractal) numerical homogenization techniques
to  illustrate our theoretical findings.
}

\section{Multiscale interface problems} \label{sec:MultiscaleInterface}
\subsection{Multiscale interface networks} \label{subsec:MIN}
Let $\bQ\subset \Rd$ be a bounded domain with Lipschitz boundary $\partial \bQ$ 
that contains  mutually disjoint interfaces $\Gamma_{k}$, $k\in \N$.
We assume that each interface $\Gamma_{k}$ is piecewise affine and
has finite $(d-1)$-dimensional Hausdorff measure.
We consider the multiscale interface network $\Gamma$ 
and its level-$K$ approximation $\Gamma^{(K)}$, given by 
\[
\Gamma=\bigcup_{k=1}^{\infty}\Gamma_{k}, 
\qquad \Gamma^{(K)}=\bigcup^{K}_{k = 1}\Gamma_{k},\quad K\in \N\, ,
\]
respectively.
For each $K\in\N$, the set
\[
\bQ\backslash\Gamma^{(K)}=\bigcup_{G\in \cG^{(K)}} G
\]
splits into mutually disjoint, open, simply connected cells $G\in \cG^{(K)}$
with the property $\partial G=\partial\overline{G}$.
The subset of invariant cells is denoted by
\begin{align*}
\cG_{\infty}^{(K)}=\left\{ G\in\cG^{(K)}\,\vert \;\,G\in\cG^{(L)}\; \forall L>K \right\} \,,
\end{align*}
and 
\begin{equation} \label{eq:NIS}
 d_{K} =\max
 \left\{ {\rm diam}\,G\,|\;G\in\cG^{(K)}\backslash\cG_{\infty}^{(K)}\right\}
\end{equation}
is the maximal size of cells $G\in\cG^{(K)}$ to be divided on higher levels.
Observe that $d_{K}\geq d_L$ holds for $L\geq K$. We assume
\begin{equation} \label{eq:DKZERO}
d_{K}\to\;0\quad
 \text{for}\quad K\to\infty\, .
\end{equation}
Denoting 
\[
(x,y)=\{x+s(y-x)\;|\; s\in (0,1)\}\;,
\]
and the number of elements of some set $M$ by $\# M\in \N\cup \{+\infty\}$, 
we also assume that 
\begin{equation} \label{eq:CDEF}
  \#(x,y)\cap \Gamma_k \leq C_k 
\end{equation}
holds for almost all $x,\;y\in \bQ$ 
with $C_k\in \N$ depending only on $k\in \N$.

\begin{example}[Cantor interface network in 3D~\cite{turcotte1994crustal}]
\label{exa:Cantor}
Consider the unit cube $\Ic=[0,1]^{3}$ in $\R^{3}$ and the canonical basis $(e_{i})_{i=1,2,3}$.
Then $\Gamma^{(K)}$, $K\in \N$, is inductively constructed as follows. 
Set $\Gamma^{(0)} = \Gamma_{0}=\partial\Ic$. 
For $k\in\N\cup\{0\}$ define 
\[
\tilde{\Gamma}_{k+1} =
\Gamma^{(k)}\cup
\left(e_{2}+\Gamma^{(k)}\right)\cup
\left(e_{3}+\Gamma^{(k)}\right)\cup
\left(e_{3}+e_{1}+\Gamma^{(k)}\right)
\cup\left(e_{2}+e_{1}+\Gamma^{(k)}\right)
\cup\left(e_{3}+e_{2}+e_{1}+\Gamma^{(k)}\right)
\]
to obtain  
\[
\Gamma_{k+1} =\left({\textstyle \frac{1}{2}}
\tilde{\Gamma}_{k+1}\right)\backslash\Gamma_{k},
\qquad \Gamma^{(K+1)}=\Gamma^{(K)}\cup \Gamma_{k+1}.
\]
Note that $\Gamma^{(K)}$ and $\Gamma= \bigcup_{k=1}^{\infty}\Gamma_{k}$  
are self-similar by construction. 
We infer $d_{K}=2^{-K}$ and  $C_{k}=2^{k-1}$. 
\end{example}

\MOD{See Figure~\ref{fig:Cantor-Sets} for an illustration of 
the Cantor interface networks $\Gamma^{(K)}$, $K=4,5,6$. 
The construction process for a 2D-analogue is illustrated in Figure~\ref{fig:illustration-cantor-levels},
where the newly added interfaces 
$\Gamma_1$, $\Gamma_2$, $\Gamma_3$, and  $\Gamma_4$ are depicted in boldface
in the  four pictures from left to right.}
\begin{figure}
\includegraphics[width=3cm]{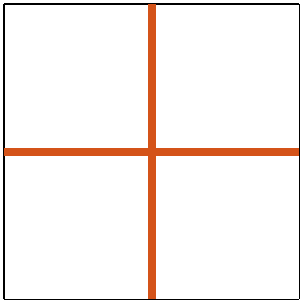}$\quad$
\includegraphics[width=3cm]{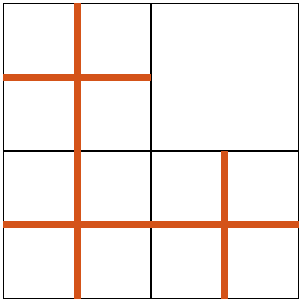}$\quad$
\includegraphics[width=3cm]{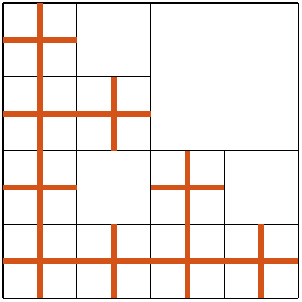}$\quad$
\includegraphics[width=3cm]{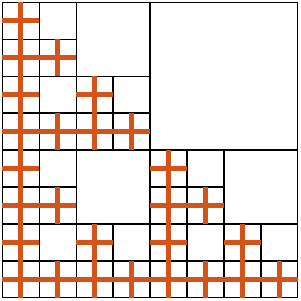}
\caption{\label{fig:illustration-cantor-levels} Construction of $\Gamma^{(K)}$,
$K=1,2,3,4$ of a Cantor interface network in 2D.} 
\end{figure}

\begin{rem}\label{rem:measure-of-Gamma}
Since all $\Gamma_k$, $k\in \N$, have Lebesgue measure zero in $\R^d$,
their countable union $\Gamma$ has Lebesgue measure zero as well. 
However, $\Gamma$ might have fractal (Hausdorff-) dimension $d-s$ for some $s\in(0,1)$
and infinite $(d-1)$-dimensional measure.
\end{rem}

\subsection{\MOD{A multiscale hierarchy of Hilbert spaces}}
For each fixed $K\in \N$, we introduce the space 
\[
\cC_{K,0}^{1}(\bQ) =\left\{ 
\MOD{ \left. v:  \overline{\bQ}\backslash\Gamma^{(K)} \to \R 
\;\;  \right\vert} \;\;v|_G\in C^1(\overline{G}) \; \forall G \in \cG^{(K)}
\text{ and } v|_{\partial\bQ}\equiv0\right\} 
\]
of piecewise  smooth functions on $\bQ\backslash \Gamma^{(K)}$.
Let $k=1,\dots,K$. As $\Gamma_k$  is piecewise affine, 
there is a normal $\nu_{\xi}$ to $\Gamma_k$ at almost all $\xi\in \Gamma_k$ 
and we fix the orientation of $\nu_{\xi}$ such that $\nu_{\xi}\cdot e_m>0$ 
\MOD{with $m = \min\{i=1,\dots,d\;|\; \nu_{\xi} \cdot e_i \neq 0 \}$, and}
$\{e_1,\dots,e_d\}$ denotes the canonical basis of $\Rd$.
\MOD{For $\xi \in \Gamma^{(K)}$ such that $\nu_{\xi}$ exists and
for $x \neq y \in \R^d$ such that $(x-y)\cdot \nu_{\xi}\neq 0$
the jump of $v\in\cC_{K,0}^{1}(\bQ)$ across $\Gamma_k$ at $\xi$
in the direction $y-x$ is defined by 
\[
\jump v_{x,y}(\xi) =\lim_{s\downarrow 0}\left(v\left(\xi+s(y-x)\right)-v\left(\xi-s(y-x)\right)\right)\; .
\]
 Up to the sign, $\jump v_{x,y}(\xi)$ 
is  equal to the normal jump of $v\in\cC_{K,0}^{1}(\bQ)$
\[
\jump v(\xi) := \jump v_{\xi-\nu_\xi,\xi+\nu_\xi}(\xi)
\]
and defined at almost all $\xi \in \Gamma_k$.} 

For some fixed \MOD{material constant $\fc>0$, that determines 
the growth of resistance to jumps with increasing fracturing,
and the geometrical constant} $C_k$ taken from \eqref{eq:CDEF}, we introduce the scalar product 
\begin{equation} \label{eq:SCALPRO}
\MOD{\left\langle v,\,w\right\rangle _{K,\fc}} =
\int_{\bQ\backslash \Gamma^{(K)}}\nabla v\cdot\nabla w\; dx + 
\sum_{k=1}^{K}\left(1+\fc\right)^{k}C_{k}\int_{\Gamma_{k}}\jump v\jump w\; d\Gamma_k\; ,
\quad  v,\; w \in \cC_{K,0}^{1}(\bQ) \;  ,
\end{equation}
with the associated norm 
\MOD{$\left\Vert v\right\Vert_{K,\fc}=\left\langle v,\,v\right\rangle _{K,\fc}^{1/2}$.}
Observe that $(1+ \fc)^k$ generates an exponential scaling of the 
resistance to jumps across $\Gamma_k$.

We set 
\begin{equation}\label{eq:K-CONST}
\cH_{K} =\text{closure}_{\left\Vert \cdot\right\Vert_{K,\fc}}\cC_{K,0}^{1}(\bQ)  
\end{equation}
to finally obtain a hierarchy of Hilbert spaces
\begin{equation} \label{eq:EMBEDDING}
 \cH_1\subset \cdots\subset \cH_{K-1} \subset \cH_K\; , \qquad  K\in \N\; ,
\end{equation}
with isometric embeddings.

\subsection{\MOD{Level-$K$ interface problems}}
\MOD{
For a given measurable function 
\begin{equation}\label{eq:ADEF}
 A:\,\Gamma\to\R
\end{equation}
satisfying
\begin{equation}\label{eq:DEFINIT}
 0<\fa\leq A(x) \leq \fA <\infty\quad \text{a.e. on }\Gamma
\end{equation}
with suitable $\fa,\fA\in\R$ and each $K\in \N$, 
we define the symmetric bilinear form
\[
a_K(v,w)= \int_{\bQ \backslash\Gamma^{(K)}} \nabla v \cdot \nabla w\; dx +
\sum_{k=1}^{K} \left(1+\fc\right)^k C_k \int_{\Gamma_k} A\,\jump{v}\jump{w}\; d\Gamma_{k} \qquad v, w \in \cH_K.
\]
For ease of presentation, we assume $\fa \leq 1 \leq \fA$ without loss of generality.
Then  $a_K(\cdot,\cdot)$
is uniformly coercive and bounded on $\cH_K$ in the sense that
\[
\fa \|v\|_{K,\fc}^2 \leq a_K(v,v),\qquad a_K(v,w)\leq \fA \|v\|_{K,\fc}\|w\|_{K,\fc}
\]
holds for all $K\in \N$.
With given functional $\ell \in \cH_K'$, $K\in \N$, from the associated dual space,
we consider the following minimization problem.
\begin{problem}[Level-$K$ interface problem]\label{prob:MultiscaleMin}
For fixed $K\in \N$, find a minimizer $u_K\in \cH_K$ of the energy functional
\begin{equation} \label{eq:Def-E-K}
\cE_K(v) = {\textstyle \frac{1}{2}} a_K(v,v) - \ell(v), \qquad v\in \cH_K.
\end{equation}
\end{problem}
The following proposition is an immediate consequence of the Lax-Milgram lemma.
\begin{prop}\label{pro:minimizer-K}
Problem~\ref{prob:MultiscaleMin} 
is equivalent to the variational problem
of finding $u_K\in \cH_K$ such that 
\begin{equation}\label{eq:prob:MultiscaleVar}
a_K(u_K,v) =\ell(v) \qquad \forall v\in \cH_K
\end{equation}
and admits a unique solution.
\end{prop}
Successive resolution of the multiscale interface network $\Gamma$ 
by level-$K$ approximations $\Gamma^{(K)}$ with increasing $K\in \N$ motivates 
investigation of the asymptotic behavior 
of finite level solutions $u_K$ for $K\to \infty$.
This  will be the subject of the next section.
}

\section{Fractal homogenization} \label{sec:Main-results}

\subsection{\MOD{Fractal function spaces}}
\MOD{
We consider the pre-Hilbert space
\[
\cH^\circ=\bigcup_{K=1}^\infty\cH_K
\]
equipped with the scalar product defined by
\[
\langle v, w \rangle_\fc = \langle v, w \rangle_{\max\{K,L\},\fc},
\qquad v\in \cH_L,\; w\in \cH_K,
\]
and associated norm $\|\cdot\|_\fc=\left\langle \cdot, \cdot \right\rangle_\fc^{1/2}$. 
A Hilbert space with dense subspace $\cH^\circ$ is obtained by classical completion.
\begin{defn}[Fractal  space] \label{def:FRACSPACE}
The fractal  space $\cH_\fc$ consists of all equivalence classes 
of Cauchy sequences $(v_{K})_{K\in \N}$ in $\cH^\circ$ 
with respect to the equivalence relation
\[
(v_{K})_{K\in \N} \sim (w_{K})_{K\in \N} \quad \Longleftrightarrow \quad 
\left\Vert v_{K}-w_{K}\right\Vert_{\fc}\to 0 \text{ for } K\to \infty\; .
\]
\end{defn}
For each Cauchy sequence in $\cH^\circ$, 
we can find an equivalent Cauchy sequence $(v_K)_{K\in \N}$ in $\cH^\circ$ 
such that $v_K\in \cH_K$, $K\in \N$,
by exploiting the hierarchy \eqref{eq:EMBEDDING}. 
We always use such a representative of elements of $\cH_\fc$.
The following result is a well-known consequence of the construction of $\cH_\fc$.
\begin{prop}\label{prop:FRACTALHILBERT}
The fractal space $\cH_\fc$ is a Hilbert space equipped with the scalar product
\begin{equation} \label{eq:SCALPRODH}
 \left\langle v,w\right\rangle_\fc
 =\lim_{K\to\infty}\left\langle v_K, w_K\right\rangle_{K,\fc}\, ,
\qquad v=(v_K)_{K\in\N},\;w=(w_K)_{K\in\N}\in \cH_\fc \; 
\end{equation}
and associated norm $\|\cdot\|_\fc=\left\langle \cdot, \cdot \right\rangle_\fc^{1/2}$.
\end{prop}
From now on, we identify the spaces $\cH_K$ 
with their isometric embeddings in $\cH_\fc$ defined by
\[
\cH_K \ni v_K \mapsto (v_L)_{L\in \N}\in \cH_\fc 
\quad \text{with } v_L=v_K,\text{ if } L \geq K \text{ and } v_L=0 \text{ else}.
\]
By construction, we have the following approximation result.
}
\begin{prop} \label{prop:HAPPROX}
For any fixed $\fc>0$, the  hierarchy
\begin{equation}
 \cH_1  \subset \cdots \subset \cH_K \subset \cdots \subset \cH_\fc
\end{equation}
consists of closed subspaces  $\cH_K$ of $\cH_\fc$, $K\in \N$,  with the property
\begin{equation} \label{eq:K-APPROX}
 \inf_{v\in \cH_K}\Vert w - v \Vert_\fc \to 0 \quad \text{for }K\to \infty\qquad \forall w \in \cH_\fc\; ,
\end{equation}
and $\bigcup_{K\in \N}\cC_{K,0}^{1}(\bQ)$ is dense in $\cH_\fc$.
\end{prop}
\MOD{
\begin{rem}
For each fixed  $K\in \N$, the spaces $\cH_K$  are independent of $\fc$.
This is no longer the case for the limit space $\cH_\fc$,
because   $v= (v_K)_{K\in \N}\in \cH_\fc$ for a certain $\fc>0$ implies that
the jumps $\|\jump{v_K}\|_{L^2(\Gamma_k)}$
are decreasing fast enough to
compensate the exponential weights $C_k(1+\fc)^k$ for this $\fc$, 
which might no longer be the case 
for larger weights $C_k(1+\fc')^k$ with some $\fc' > \fc$ 
so that $v \not\in \cH_{\fc'}$.
\end{rem}
}
\MOD{From now on, we will mostly skip the subscript $\fc$
 for notational convenience.}
A more intuitive representation 
of the scalar product $\langle\cdot,\cdot\rangle$ in $\cH$ 
and its associated norm $\|\cdot\|$ in terms of generalized jumps and gradients
will be derived in Section \ref{subseq:char-scrH} below.

\subsection{Sobolev embeddings}
\label{sec:Poincare-inequalities}
We now investigate the embedding of the fractal space $\cH$ into the fractional Sobolev spaces $H^s(\bQ)$, $s\in (0,\frac12)$,
equipped with the Sobolev-Slobodeckij norm
\[
\Vert v \Vert_{H^s(\bQ)}=\left(\int_{\bQ}|v|^2\; dx + \int_{\bQ}\int_{\bQ}\frac{|v(x)-v(y)|^2}{|x-y|^{d+2s}}\; dx dy\right)^{\frac{1}{2}}\; .
\]
\begin{lem}\label{lem:xyDiff}
Let $K \in \N$, $v\in\cC_{K,0}^{1}(\bQ)$, and $x \neq y\in \bQ$. 
Then the following inequality holds \MOD{for every $\fc>0$ and} for a.e. $x,y\in\Rd$ 
\begin{equation}\label{eq:fundamental-estimate}
\begin{array}{rl}
\left|v(x)-v(y)\right|^{2} & \displaystyle
\leq \left(1 + {\textstyle \frac{1}{\fc}}\right)
\left|x-y\right|^{2}\int_{0}^{1}\left|\nabla v\left(x+s(y-x)\right)\right|^{2}\;ds \\
& \qquad + \displaystyle \left(1 + {\textstyle \frac{1}{\fc}}\right)
\sum_{k=1}^{K}\left(1+\fc\right)^{k}C_{k}\sum_{\xi\in(x,y)\cap\Gamma_{k}}
\MOD{\jump v^{2}_{x,y}}(\xi)\,,
\end{array}
\end{equation}
\MOD{where  $\nabla v(x+s(y-x))$ is understood to be zero, if $x+s(y-x)\in \Gamma^{(K)}$.}
\end{lem}
\begin{proof}
\MOD{Let $x,y$ be such that $(x,y)\cap\Gamma^{(K)}$ is finite.} Using the Cauchy-Schwarz inequality and 
the binomial estimate $2ab<\frac{1}{\fc}a^{2}+\fc b^{2}$ with $\fc >0$ and $a,b\in \R$, we infer
\begin{align*}
\left|v(x)-v(y)\right|^{2} & 
\leq\left(\sum_{k=1}^K\sum_{\xi\in(x,y)\cap\Gamma_{k}}\jump v_{x,y}(\xi)+
\int_{0}^{1}\nabla v\left(x+s(y-x)\right)\cdot\left(y-x\right)\;ds\right)^{2}\\
 & \leq \left(1 + {\textstyle \frac{1}{\fc}}\right)
 \left|x-y\right|^{2}\int_{0}^{1}\left|\nabla v\left(x+s(y-x)\right)\right|^{2}\;ds+
 \left(1+\fc\right)
 \left(\sum_{k=1}^K\sum_{\xi\in(x,y)\cap\Gamma_{k}}\jump v_{x,y}(\xi)\right)^{2}\\
 & \leq \left(1+{\textstyle\frac{1}{\fc}}\right)
 \left|x-y\right|^{2}\int_{0}^{1}\left|\nabla v\left(x+s(y-x)\right)\right|^{2}\;ds \\
 & \qquad + \left(1+\fc\right)\left(1+{\textstyle\frac{1}{\fc}}\right)
 \left(\sum_{\xi\in(x,y)\cap\Gamma_{1}}\jump v_{x,y}(\xi)\right)^{2}
  +\left(1+\fc\right)^{2}
 \left(\sum_{k=2}^{K}\sum_{\xi\in(x,y)\cap\Gamma_{k}}\jump v_{x,y}(\xi)\right)^{2}\,.
\end{align*}
According to the Cauchy-Schwarz inequality and the definition of $C_{k}$ in \eqref{eq:CDEF}, we have
\[
\left(\sum_{\xi\in(x,y)\cap\Gamma_{k}}\jump v_{x,y}(\xi)\right)^{2}\leq C_{k}\sum_{\xi\in(x,y)\cap\Gamma_{k}}\jump v^{2}_{x,y}(\xi)
\]
and the assertion follows by induction.
\end{proof}

We are ready to state the main result of this subsection.
\begin{thm}\label{thm:embedding}
The continuous embeddings
\begin{equation}
\cH_\fc \subset L^2(\bQ) \qquad \text{and} \qquad \cH_\fc \subset H^s(\bQ)
\end{equation}
hold \MOD{for every $\fc>0$} and every $s\in[0,\frac{1}{2})$.
\MOD{In particular, the following Poincar\'e-type inequality
\begin{equation}\label{eq:Poincare-H-1-2}
\left\Vert v\right\Vert _{L^2(\bQ)}^{2}           
\leq C_0 
\left(
\left\Vert \nabla v\right\Vert _{L^{2}(\bQ\backslash\Gamma)}^{2} +
\sum_{k=1}^{\infty}\left(1+\fc\right)^{k}C_{k}\Vert\jump{v}\Vert_{L^2(\Gamma_k)}^{2}
\right)\,,
\end{equation} 
holds with 
$C_0=\left(1 + {\textstyle \frac{1}{\fc}}\right)\mathrm{diam}(\bQ)\max\{{\mathrm{diam}}(\bQ),1\}$.}
\end{thm}
\begin{proof}
We use an approach introduced by Hummel~\cite{Hummel1999}. 
Let $K\in\N$, $v\in\cC_{K,0}^{1}(\bQ)$, and $k=1,\dots,K$.
\MOD{We extend $v$ by zero to a function $v:\Rd\to\R$,
fix some $\eta>0$ to be specified later,}
and consider the orthonormal basis $(e_{i})_{i=1,\dots,d}$ of $\Rd$. 
Exploiting that the determinant $g_k$ of the first fundamental form of $\Gamma_k$  
satisfies $g_k\geq 1$, we obtain
\begin{align*}
\int_{\bQ}\sum_{\xi\in(x,x+\eta e_{1})\cap\Gamma_{k}}\jump v^{2}_{(x,x+\eta e_{1})}(\xi)\,\d x 
 & \leq \int_{\R}\left(\int_{\R^{d-1}}\sum_{\xi\in(x,x+\eta e_{1})\cap\Gamma_{k}}\jump v^{2}_{(x,x+\eta e_{1})}(\xi)\sqrt{g_k}\,\d x_{2}\dots\d x_{d}\right)\d x_{1}\\
 & \leq\int_{\R}\left(\int_{\Gamma_{k}\cap\left(\left(x_{1},x_{1}+\eta\right)\times\R^{d-1}\right)}\jump v^{2}_{(x,x+\eta e_{1})}(\xi)\,\d\Gamma_k \right)\d x_{1}\\
 & =\int_{\Gamma_{k}}\ \left(\int_{\xi_1 - \eta}^{\xi_1}\jump v^{2}_{(x,x+\eta e_{1})}\,\d x_{1}\right)(\xi)\, \d\Gamma_k
   = \eta\int_{\Gamma_{k}}\jump v^{2}(\xi)\,\d\Gamma_k\; , 
\end{align*}
where we used that \MOD{$\jump v^{2}(\xi)$ is well defined a.e. on $\Gamma_k$ and}
$\xi =(\xi_1, \xi')\in \Gamma_k\cap (\left(x_{1},x_{1}+\eta\right)\times \R^{d-1})$ 
is equivalent to $x_1\in(\xi_1 -\eta,\xi_1)$ with $\xi=(\xi_1, \xi') \in \Gamma_k$.
The same arguments provide
\begin{equation}\label{eq:JUMPES}
 \int_{\bQ}\sum_{\xi\in(x,x+\eta e)\cap\Gamma_{k}}\jump v^{2}(\xi)\,\d x
\leq \eta\int_{\Gamma_{k}}\jump v^{2}\,\d\Gamma_k
\end{equation}
for any unit vector $e\in \R^d$. 
Inserting  \eqref{eq:JUMPES} after integrating \eqref{eq:fundamental-estimate} 
with  $y=x+\eta e$ over $\bQ$ leads to
\begin{equation}\label{eq:poincare-shift-eta}
\int_{\bQ}\left|v(x)-v(x+\eta e)\right|^{2}dx\leq
\eta\left(1 + {\textstyle \frac{1}{\fc}}\right)
\left(
\eta \left\Vert \nabla v\right\Vert _{L^{2}(\bQ\backslash\Gamma^{(K)})}^{2} +
\sum_{k=1}^{K}\left(1+\fc\right)^{k}C_{k}\Vert\jump{v}\Vert_{L^2(\Gamma_k)}^{2}
\right)\,.
\end{equation}
We select $\eta \geq {\mathrm{diam}}(\bQ)$
to obtain  \MOD{the Poincar\'e-type inequality \eqref{eq:Poincare-H-1-2} 
and thus $\cH_\fc \subset L^2(\bQ)$.}

Next, we divide \eqref{eq:poincare-shift-eta} by $|\boldsymbol{\eta}|^{d+2s}$ 
and integrate over 
\[
 \bQ \subset \{\eta e\;|\; \eta \leq {\mathrm{diam}}(\bQ),\; e \in \Sphere\},
\]
where $\Sphere$ denotes the unit sphere in $\R^d$, to find that
\begin{equation}
\left\Vert v\right\Vert _{H^{s}(\bQ)}^{2}\leq
\left(1 + {\textstyle \frac{1}{\fc}}\right)C_{s}
\left(
\left\Vert \nabla v\right\Vert _{L^{2}(\bQ\backslash\Gamma^{(K)})}^{2} +
\sum_{k=1}^{K}\left(1+\fc\right)^{k}C_{k}\Vert\jump{v}\Vert_{L^2(\Gamma_k)}^{2}
\right)
\end{equation} 
holds for all $v\in\cC_{K,0}^{1}(\bQ)$ and all $K\in \N$ 
with $C_s=\max\{{\mathrm{diam}}(\bQ),1\}|{\Sphere}|\int_0^{{\mathrm{diam}}(\bQ)}\eta^{-2s}\d \eta <\infty$ for every $s \in [0, \frac{1}{2})$. 
By Proposition~\ref{prop:HAPPROX}, 
the subspace $\bigcup_{K \in \N}\cC_{K,0}^{1}(\bQ)$ is dense in $\cH$. 
This concludes the proof.
\end{proof}

\begin{rem}\label{rem:HSCONV}
For any given $(v_K)_{K \in \N}\in \cH$, 
there is a unique $v \in \bigcap_{0 < s < \frac12}H^s(\bQ)$ such that 
\begin{equation} \label{eq:SID}
\|v -v_K\|_{H^s(\bQ)}\to 0 \quad  for \quad K\to \infty\qquad \forall s \in (0,\textstyle{\frac{1}{2}})
\end{equation}
as a consequence of Theorem~\ref{thm:embedding}. 
\end{rem}

\subsection{Weak gradients and generalized jumps}\label{subseq:char-scrH}

Let $(v_K)_{K\in\N}\in\cH$ and observe that 
\[
 \bQ\backslash \Gamma = \bQ\cap (\bigcup_{k=1}^{\infty} \Gamma_k)^{\complement}\subset \bQ\backslash \Gamma^{(K)} 
\]
is Lebesgue measurable. Hence, we have
\[
\|\nabla v_K\|_{L^2(\bQ\backslash \Gamma)}^2 +
\sum_{k=1}^{K} (1+ \fc)^k C_k \|\jump{v_K}\|_{L^2(\Gamma_k)}^2
\leq \|v_K\|_K^2\; \qquad \forall K\in \N\; .
\]
Therefore, $(\nabla v_K)_{K\in \N}$ and $(\jump{v_K})_{K\in \N}$ are Cauchy sequences
in $L^2(\bQ\backslash \Gamma)^d$ and in the sequence space $(L^2(\Gamma_k))_{k \in \N}$
equipped with the weighted norm
\[
\Vert j \Vert_{\Gamma}= 
\left(\sum_{k=1}^{\infty} (1+ \fc)^k C_k \| j_k \|_{L^2(\Gamma_k)}^2\right)^{\frac{1}{2}}\; ,
\quad j= (j_k)_{k \in \N}\in (L^2(\Gamma_k))_{k \in \N} \; ,
\]
respectively. 
In light of the completeness of $L^2(\bQ\backslash \Gamma)^d$ and of $(L^2(\Gamma_k))_{k \in \N}$,
this leads to the following definition.

\begin{defn}\label{def:WEAKGEN}
 Let $(v_K)_{K\in\N}\in\cH$ with associated $v \in \bigcap_{0 < s < \frac12}H^s(\bQ)$ 
 that is characterized by \eqref{eq:SID}. Then the limits
 \[
\nabla v = \lim_{K\to \infty} \nabla v_K \quad \text{in }L^2(\bQ\backslash \Gamma) 
\quad\text{and}\quad
\jump v = \lim_{K\to \infty}\jump {v_K} \quad \text{in }(L^2(\Gamma_k))_{k \in \N}
 \]
 are called the weak gradient and generalized jump of $v$, respectively.
\end{defn}

Since  the fractal (and Hausdorff-) dimension of $\Gamma$ 
might be larger than $d-1$, 
it is not obvious to define $L^{2}(\Gamma)$ 
(and to infer convergence of $(\jump{u_{K}})_{K\in \N}$ in $L^{2}(\Gamma)$),
because it is not obvious which measure to choose.

\begin{prop}
Let $(v_K)_{K\in\N}\in\cH$ with associated $v\in \bigcap_{0 < s < \frac12}H^s(\bQ)$ 
that is characterized by \eqref{eq:SID}.
Then the weak gradient $\nabla v$ and the generalized jump $\jump v$ of $v$ 
are related by the identity 
\begin{equation}\label{eq:WEAK}
\int_{\bQ} v \nabla\cdot\varphi \; dx 
=-\int_{\bQ\backslash\Gamma} \nabla v\cdot\varphi\; dx
+\sum_{k=1}^\infty \int_{\Gamma_k}\jump v \varphi\cdot\nu_k\; d\Gamma_k
\qquad \forall \varphi \in C_0^{\infty}(\R^d)^d\; .
\end{equation}
\end{prop}
\begin{proof}
Let $\varphi \in C_0^{\infty}(\R^d)^d$
and recall that $\Gamma$ has Lebesgue measure zero in $\R^d$
according to Remark \ref{rem:measure-of-Gamma}. 
As a consequence, we have
\begin{align*}
 \int_{\bQ\backslash \Gamma^{(K)}}\nabla v_K\cdot \varphi\; dx =
\int_{\bQ\backslash \Gamma}\nabla v_K\cdot \varphi\; dx 
+ \int_{\Gamma\backslash \Gamma^{(K)}} \nabla v_K\cdot \varphi\; dx
\;\;\to \;\; 
\int_{\bQ\backslash \Gamma} \nabla v\cdot \varphi\; dx
\quad \text{for }K\to \infty
\end{align*}
which by Definition~\ref{def:WEAKGEN} leads to
\begin{align*}
\int_{\bQ} v \nabla\cdot\varphi \; dx & =\lim_{K\to\infty}\int_{\bQ} v_K\nabla\cdot\varphi\; dx\\
& =\lim_{K\to\infty}\left(-\int_{\bQ\backslash\Gamma^{(K)}} \nabla v_K\cdot\varphi\; dx
+\sum_{k=1}^K \int_{\Gamma_k}\jump{v_K}\varphi\cdot\nu_k\; d\Gamma_k \right)\\
& =-\int_{\bQ\backslash\Gamma} \nabla v\cdot\varphi\; dx
+\sum_{k=1}^\infty \int_{\Gamma_k}\jump v \varphi\cdot\nu_k\; d\Gamma_k\; .
\end{align*}
\end{proof}
\begin{thm}\label{thm:structure}
Let $v_{\cH}=(v_K)_{K\in\N}, w_{\cH}=(w_K)_{K\in\N}\in\cH$ 
with associated $v,\; w \in \bigcap_{0 < s < \frac12}H^s(\bQ)$
that are characterized by \eqref{eq:SID}.
Then we have
\begin{equation}  \label{eq:scrH-inner-product}
\langle v_\cH,\,w_\cH \rangle =
\int_{\bQ\backslash\Gamma}\nabla v\cdot\nabla w\; dx+ 
\sum_{k=1}^{\infty}\left(1+\fc\right)^{k}C_{k}\int_{\Gamma_{k}}\,\jump v\jump w\; d\Gamma_k\,.
\end{equation}
\end{thm}
\begin{proof}
By Definition~\ref{def:WEAKGEN} of generalized jumps, we have
\[
\sum_{k=1}^{K}\left(1+\fc\right)^{k}C_{k}\int_{\Gamma_{k}}\,\jump{v_K}\jump{ w_K}\; d\Gamma_k
\;\; \to \;\; \sum_{k=1}^{\infty}\left(1+\fc\right)^{k}C_{k}\int_{\Gamma_{k}}\,\jump v\jump w\; d\Gamma_k
\quad \text{for }K\to \infty
\]
and 
as $\Gamma$ has Lebesgue measure zero in $\R^d$ (cf.~Remark \ref{rem:measure-of-Gamma}),
we obtain
\[
\int_{\bQ\backslash \Gamma^{(K)}}\nabla v_K \cdot \nabla w_K\; dx 
=
\int_{\bQ\backslash \Gamma}\nabla v_K \cdot \nabla w_K\; dx
\;\; \to \;\; \int_{\bQ\backslash \Gamma}\nabla v \cdot \nabla w\; dx\quad \text{for }K\to \infty\; .
\]
This concludes the proof.
\end{proof}
From now on, we  identify $(v_K)_{K\in\N}\in \cH$
with $v \in \bigcap_{0 < s < \frac12}H^s(\bQ)$  characterized by \eqref{eq:SID} 
and use the representation \eqref{eq:scrH-inner-product} 
of the scalar product $\langle \cdot,\cdot\rangle$ in $\cH$.

For the Cantor interface network, cf.~Example~\ref{exa:Cantor},
the weighting factors $(1+\fc)^{k}C_{k}$ in \eqref{eq:scrH-inner-product}
are exponentially increasing with $k$,
causing exponentially decreasing generalized jumps accross~$\Gamma_k$.

\subsection{\MOD{Fractal interface problems}}\label{sub:FracIntProb}
\MOD{We consider the functional 
\[
\ell(v) = \int_{\bQ} f v\; dx
\]
with some given $f \in L^2(\bQ)$. 
Note that the Poincar\'e-type inequality \eqref{eq:Poincare-H-1-2} implies
$\ell \in \cH' \subset \cH_K'$ for all $K \in \N$. 
The solutions $u_K$ of the level-$K$ interface 
Problems~\ref{prob:MultiscaleMin} for $K\in \N$
then satisfy the uniform stability estimate
\begin{equation}\label{eq:GALSTAB}
\|u_K\|\leq C_0\fa^{-1}\|f\|_{L^2(\bQ)}, \qquad K\in \N,
\end{equation}
with the constant $C_0$ appearing in \eqref{eq:Poincare-H-1-2}.
}

\MOD{
We define the symmetric bilinear form
\begin{equation} \label{eq:BLF} 
a(v,w)= \int_{\bQ \backslash\Gamma} \nabla v \cdot \nabla w\; dx +
\sum_{k=1}^{\infty} \left(1+\fc\right)^k C_k \int_{\Gamma_k} A\,\jump{v}\jump{w}\; d\Gamma_{k}, \quad v, w \in \cH,
\end{equation}
with $A: \Gamma \to \R$ taken from \eqref{eq:ADEF}.
Note that $a(\cdot,\cdot)$ is well-defined, coercive and bounded 
in light of Definition~\ref{def:WEAKGEN}
and assumption \eqref{eq:DEFINIT}. 
Now, we are ready to 
formulate an asymptotic limit of the level-$K$ interface Problems~\ref{prob:MultiscaleMin} for $K\to \infty$. 
\begin{problem}[Fractal interface problem]\label{prob:FractalMin}
Find a minimizer $u\in \cH$ of the energy functional
\[
\cE(v) =  {\textstyle \frac{1}{2}} a(v,v) - \ell(v), \qquad v\in \cH.
\]
\end{problem}
In light of of Proposition~\ref{prop:HAPPROX},
the following existence and approximation result is a consequence of the Lax-Milgram lemma and C\'ea's lemma.
}

\MOD{\begin{thm}\label{thm:FractalMin}
Problem~\ref{prob:FractalMin}  is equivalent to the variational problem
of finding $u\in \cH$ such that 
\begin{equation}\label{eq:prob:FractalVar}
a(u,v)=\ell(v) \qquad \forall v\in \cH
\end{equation}
and admits a unique solution. Moreover,  the error estimate
\begin{equation}\label{eq:Cea}
 \Vert u - u_K \Vert \leq 
 {\textstyle \frac{\fA}{\fa} }\inf_{v \in \cH_K} \Vert u - v \Vert
\end{equation}
implies convergence $\Vert u - u_K \Vert \to 0$ for $K\to \infty$.
\end{thm}
In the next section, we will improve the straightforward error estimate \eqref{eq:Cea}
under more restrictive assumptions on the geometry of the multiscale interface network.
}

\section{\MOD{Exponential} Error estimates}\label{sec:error-estimates}
%
%
We concentrate on the special case that all cells $G\in \cG^{(K)}$, $K\in \N$, 
are hyper-cuboids with edges $e_{G,i}$, \MOD{$i=1,\dots, d2^{d-1}$}, 
either parallel or perpendicular to the unit vectors $e_i$, $i=1,\dots, d$.
\MOD{For $K\in \N$, we set 
\[
d_G^{\max}=\max_{i}|e_{G,i}|,
\quad d_G^{\min}=\min_{i}|e_{G,i}|\; ,\quad G\in \cG^{(K)}\; ,
\qquad \text{and} \quad d_K^{\min} = \min_{G \in \cG^{(K)}}d_G^{\min}\; ,
\]
and assume that there is a constant $\fg >0$ such that
\begin{equation}\label{eq:SHAPEREG}
d^{-1/2} d_K \leq d_G^{\max} \leq d^{-1/2} \fg d_G^{\min} 
\qquad \forall G\in \cG^{(K)},\;  K\in \N\; ,
\end{equation}
with space dimension $d$ and $d_K$ taken from \eqref{eq:NIS}.
Note that \eqref{eq:SHAPEREG} implies uniform shape regularity of all $G\in \cG^{(K)}$
together with quasi-uniformity of the partition $\cG^{(K)}\setminus \cG^{(K)}_\infty$.
}
We also assume that $\cG^{(K)}$ is regular for all $K\in \N$ 
in the sense that two cells $G\in \cG^{(K)}\setminus \cG^{(K)}_{\infty}$ and $G' \in \cG^{(K)}$
have an intersection $F=G\cap G'$ with non-zero $(d-1)$-dimensional Hausdorff measure,
if and only if $F$ is a common $(d-1)$-face of $G$ and $G'$.
Note that the Cantor set described in subsection~\ref{subsec:MIN} satisfies 
both of these additional assumptions.


The derivation of error estimates will rely on a representation of
the residual of the approximate solution $u_K$ of Problem~\ref{prob:MultiscaleMin} 
in terms of its normal traces on $\Gamma_L$, $L > K$
(cf., e.g., variational formulations 
of substructuring methods \cite{quarteroni1999domain}). 
This requires additional regularity in a neighborhood of $\Gamma_L$, $L > K$.
\begin{lem} \label{lem:REGULARITY}
 Let $K \in \N$, $G\in \cG^{(K)}\setminus \cG^{(K)}_{\infty}$, $L > K$, 
 and $\Gamma_L \cap G = \bigcup_{i=1}^d \gamma_{L,G,i}$,
 such that $e_i \perp \gamma_{L,G,i}$, $i=1,\dots,d$.
 Then, for each $i=1,\dots,d$ 
 there are open sets $U_{L,G,i}\subset G$ with $\gamma_{L,G,i} \subset U_{L,G,i}$
 such that $\partial_i u_K \in H^1(U_{L,G,i})$ and the a priori estimate
 \begin{equation} \label{eq:TRACEBOUND}
  d_L\|\partial_i u_K\|^2_{L^2(\gamma_{L,G,i})} 
  \leq c\left(d_L^2 \|f\|^2_{L^2(G_i^*)} + \|\partial_i u_K\|^2_{L^2(G_i^*\setminus \Gamma^{(K)})}\right)
 \end{equation}
holds with a constant $c$ depending only on $\fg$ and $d$.
\end{lem}

\begin{proof}
\MOD{The main idea of the proof is to
first provide local a priori $H^1$-bounds  
for difference quotients 
$D_i^h u_K = {\textstyle\frac{1}{h}}\left(u_K (\cdot + e_i h) - u_K \right)$ 
that are uniform in $h$ on suitable subsets $U_{L,G,i}$.
These $H^1$-bounds then lead to related $H^1$-bounds for $\partial_i u_K$
by well-known arguments from Evans~\cite{Evans1998}
so that the desired a priori  estimates \eqref{eq:TRACEBOUND} 
finally follow from the trace theorem.
Most part the proof is devoted to the local a priori $H^1$-bounds 
for $D_i^h u_K$.
They are derived from the weak formulation \eqref{prob:MultiscaleMin} of the problem
by inserting test functions of the form
$v= -D_i^{-h}(\xi^2 D_i^{h}) u_K\in \cH_K$
with sophistically constructed smooth functions $\xi=\xi_{L,G,i}$ 
with local support in some suitable $U_{L,G,i}^*$ and $\xi\equiv 1$ 
on the final subset $U_{L,G,i} \subset U_{L,G,i}^*$.
}

\begin{figure}
    \hspace*{0.7cm}
    \def\svgwidth{15cm}
    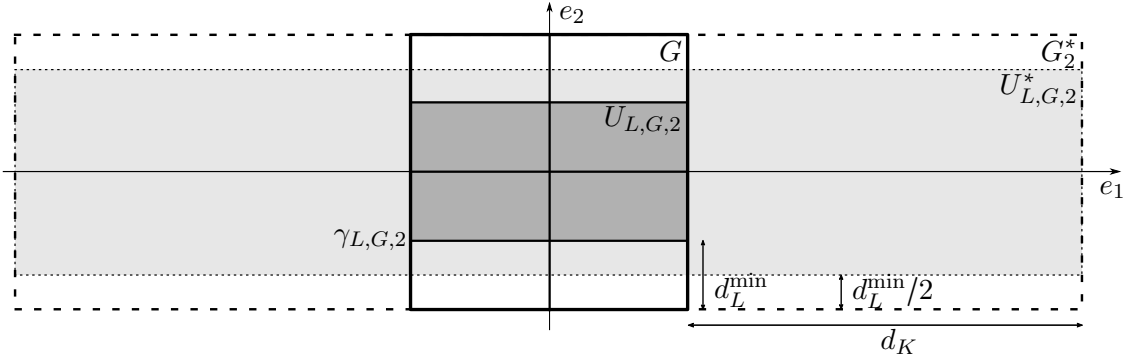
    \caption{Construction of $\xi_{L,G,i}$: 
     $G \in \cG^{(K)}\setminus\cG^{(K)}_{\infty}$  with 
     $\Gamma^{(L)}\cap G=\gamma_{L,G,1} \cup \gamma_{L,G,2}$ and $G_2^*$ (dashed),
    light-grey $U^*_{L,G,2}$, and dark-grey $U_{L,G,2}$.
    \label{fig:sketch-lem-3-1}}
\end{figure}
Let $G=(-g,g)^d \in \cG^{(K)}\setminus \cG^{(K)}_{\infty}$, 
for simplicity, and consider some fixed $i=1,\dots,d$.
\MOD{We start with the construction of $\xi_{L,G,i}$ 
which is illustrated in Figure~\ref{fig:sketch-lem-3-1}.
Note that $\gamma_{L,G,1}=\{(0,s)\;|\; s\in (-g,g)\}$  
and $\gamma_{L,G,2}=\{(s,k\frac{g}{2}\;|\;k=1,0,-1,s\in(-g,g)\}$ in this illustration.}
We select $\xi_L \in C_0^{\infty}(\R)$ 
with support in $[-g+d_L^{\min}/2,g-d_L^{\min}/2]$ and the properties
$0\leq \xi_L(x) \leq 1$ $\forall x\in \R$,
$\xi_L(x) = 1$ if $|x|\leq g-d_L^{\min}$, and 
$\xi_L'(x) \leq 2(d_L^{\min})^{-1}\leq 2 \fg d_L^{-1}$.
We further select $\xi_G \in C_0^{\infty}(\R^d)$ with support in $G_i^*$,
\[
G_i^*=\left\{x\in\Rd\,|\;\exists y\in G:\,|x-y|<d_K,\,(x-y)\cdot e_i=0\right\},
\]
satisfying
$0\leq \xi_G(x) \leq 1$ for all $x\in \R^d$,
$\xi_G(x) = 1$ for all $x \in G$ with $|x_i|\leq g-d_L^{\min}$, 
and $|\nabla \xi_G(x)| \leq (d_K^{\min})^{-1}\leq \fg d_K^{-1}\leq \fg d_L^{-1}$
for all $x\in \R^d$.
\MOD{We finally set $\xi_{L,G,i}(x)=\xi_L(x_i)\xi_G(x)$ for $x \in \R^d$, 
\[
U_{L,G,i}^* = \text{int supp }\xi_{L,G,i} \subset G_i^*\; , \quad \text{and}  \quad
U_{L,G,i}=\text{int } \{x \in G\;|\; \xi(x)=1\} \subset U_{L,G,i}^*\; .
\]
}
For notational convenience, 
\MOD{we mostly write $U^\ast=U^\ast_{L,G,i}$ and $\xi=\xi_{L,G,i}$ in the sequel.} 
Note that
\begin{equation} \label{eq:XIDIFF}
|\nabla \xi |\leq |\xi_L \nabla \xi_G| + |\xi_L'  \xi_G| \leq 3  \fg d_L^{-1}\; .
\end{equation}

Extending $v\in \cH$ from $\bQ$ to $\R^d$ by zero, we define
\[
D^h_i v ={\textstyle\frac{1}{h}}\left(v (x + e_i h) - v(x) \right)\; , \quad v\in \cH\; ,
\]
with $|h|>0$. 
Let $h>0$ be sufficiently small to provide  $-D_i^{-h}(\xi^2 D_i^h u_K)\in \cH_K$. 
Then \eqref{eq:prob:MultiscaleVar} yields
\begin{equation} \label{eq:FANCYTEST}
 a(u_K,-D_i^{-h}(\xi^2 D_i^h u_K)) =\ell(-D_i^{-h}(\xi^2 D_i^h u_K))\; .
\end{equation}
Exploiting
\begin{equation}\label{eq:SHIFT}
\Gamma^{(K)}\cap (h e_i + U^*) \subset \Gamma^{(K)}\cap G_i^*
\end{equation}
for sufficiently small $|h|>0$, we get
\[
\begin{array}{l}
\displaystyle \int_{\bQ\setminus\Gamma^{(K)}}
\nabla u_K \cdot \nabla (-D_i^{-h}(\xi^2 D_i^h u_K))\; dx  \\
 \qquad    \qquad\qquad 
= \displaystyle\int_{U^*\setminus\Gamma^{(K)}}|\nabla D_i^h u_K|^2 \xi^2\; dx 
+ \int_{U^*\setminus\Gamma^{(K)}} \nabla D_i^h u_K \cdot \nabla (\xi^2) D_i^h u_K\; dx\; .
\end{array}
\]
Similarly, \eqref{eq:SHIFT} leads to
\[
\int_{\Gamma_k} A \jump{u_K}\jump{-D_i^{-h}(\xi^2 D_i^h u_K)}\; d\Gamma_k =
\int_{U^*\cap \Gamma_k} A \jump{D_i^h u_K}^2\xi^2\; d\Gamma_k
\]
for all $k= 1,\dots K$. Utilizing \eqref{eq:SHIFT}, 
the fundamental theorem of calculus and a density argument, it can be shown that
\begin{equation}\label{eq:DEST}
\int_{U^*}|D_i^{-h} v|^2\; dx 
\leq \int_{U^*\setminus \Gamma^{(K)}}|\partial_i v|^2\; dx \quad \forall v\in \cH\; .
\end{equation} 
Together with the Cauchy-Schwarz inequality and $U^*\subset G_i^*$ this leads to
\[
\begin{array}{rcl}
\displaystyle |\ell(-D_i^{-h}(\xi^2 D_i^h u_K))| &\leq &
\|f\|_{L^2(G_i^*)}\|D_i^{-h}(\xi^2 D_i^h u_K)\|_{L^2(U^*)} \\
 &\leq& \displaystyle  \|f\|_{L^2(G_i^*)}
\left(\int_{U^*\setminus \Gamma^{(K)}} |\partial_i(\xi^2 D_i^h u_K)|^2\; dx\right)^{1/2}\; .
\end{array}
\]
We insert the above identities and  this estimate into \eqref{eq:FANCYTEST} to obtain 
\[
\begin{array}{l}
\displaystyle \int_{U^*\setminus\Gamma^{(K)}} |\nabla D_i^h u_K|^2 \xi^2 \; dx +
\sum_{k=1}^{K} \left(1+\fc\right)^k C_k \int_{U^*\cap\Gamma_k} A\,\jump{D_i^h u_K}^2 \xi^2
\; d\Gamma_{k}  \\
\qquad \qquad \leq \displaystyle 
\|f\|_{L^2(G_i^*)}
\left(\int_{U^*\setminus \Gamma^{(K)}} |\partial_i(\xi^2 D_i^h u_K)|^2\; dx\right)^{1/2} +
\int_{U^*\setminus \Gamma^{(K)}} |\nabla D_i^h u_K|\; |\nabla (\xi^2)|\; |D_i^h u_K| \; dx \; .
\end{array}
\]
Now \eqref{eq:XIDIFF} and multiple applications of Young's inequality yield 
\[
\begin{array}{rl}
\displaystyle \int_{U^*\setminus \Gamma^{(K)}} |\nabla D_i^h u_K|^2 \xi^2\; dx \leq &
2  \|f\|^2_{L^2(G_i^*)} 
+ 9\fg^2 d_L^{-2}\|\xi D_i^h u_K\|^2_{L^2(G_i^*)} +36 \fg^2 d_L^{-2} \|D_i^h u_K\|^2_{L^2(G_i^*)} \\
 & \qquad 
 + \textstyle{\frac{1}{4}}\displaystyle \int_{U^*\setminus \Gamma^{(K)}} |\nabla D_i^h u_K|^2 \xi^4\; dx
 + \textstyle{\frac{1}{4}}\displaystyle \int_{U^*\setminus \Gamma^{(K)}} |\nabla D_i^h u_K|^2 \xi^2\; dx
 \; .
\end{array}
\]
Utilizing $\xi^4 \leq \xi^2\leq 1$ and \eqref{eq:DEST}, this leads to
\[
\int_{U_{L,G,i}} |\nabla D_i^h u_K|^2 \; dx  \leq
c \left( \|f\|^2_{L^2(G_i^*)} + d_L^{-2} \|\partial_i u_K \|^2_{L^2(G_i^* \setminus \Gamma^{(K)})}\right).
\]
Now the desired regularity $\partial_i u_K \in H^1(U_{L,G,i})$ 
and the corresponding a priori estimate
\begin{equation} \label{eq:APDER}
 \|\nabla \partial_i u_K\|_{L^2(U_{L,G,i})}^2 \leq 
 c \left( \|f\|_{L^2(G_i^*)} + d_L^{-2} \|\partial_i u_K \|^2_{L^2(G_i^* \setminus \Gamma^{(K)})}\right)
\end{equation}
are a consequence of~\cite[Chapter 5.8.2, Theorem 3]{Evans1998}.

It remains to show the a priori bound \eqref{eq:TRACEBOUND}.
Let $i=1,\dots,d$ be fixed, $\gamma=\gamma_{L,G,i}$, and 
$G_{\gamma}\in \cG^{(L)}$ such that $\gamma$ is a $(d-1)$-face of $G_{\gamma}$.
Utilizing affine transformations of
$G_{\gamma}\cap U_{L,G,i}$ and $\gamma$ to the reference domains $(0,1)^d$
and $(0,1)^{d-1}\times \{0\}$, respectively, we obtain
\[
\int_{\gamma} |v|^2\; d \gamma 
\leq C \fg^d\left(d d_L^2 \|\nabla v\|^2_{L^2(G_{\gamma}\cap U_{L,G,i})} +
\|v\|^2_{L^2(G_{\gamma}\cap U_{L,G,i})}\right)\quad \forall v\in H^1(G_{\gamma}\cap U_{L,G,i})
\]
with the generic constant $C$ emerging from the trace theorem on $(0,1)^d$.
Now  \eqref{eq:TRACEBOUND} follows
by inserting $v=\partial_i u_K$ and utilizing the a priori estimate \eqref{eq:APDER}. 
\end{proof}

After these preparations, we are ready to state the main result of this section.
\begin{thm} \label{thm:ERREST}
For each $K\in \N$, the approximate solution $u_K$ of Problem~\ref{prob:MultiscaleMin}
satisfies the error estimate
\begin{equation}
 \|u - u_K\|^2 \leq C\left( \sup_{k > K}C_k^{-1}d_k^{-1} \right)
 \|f\|_{L^2(\bQ)}^2
(1+\fc)^{-K}
\end{equation}
\MOD{with $C$ only depending on the space dimension $d$, 
shape regularity $\fg$ in \eqref{eq:SHAPEREG}, coercivity $\fa$ in \eqref{eq:DEFINIT}, 
the Poincar\'e-type constant $C_0$ in \eqref{eq:Poincare-H-1-2} and on the material constant $\fc$ in \eqref{eq:SCALPRO}}.
\end{thm}
\begin{proof}
\MOD{For $u \neq u_K$ we get the residual error estimate
\[
\|u - u_K\| \leq   \fa^{-1} r_K(u-u_K)/\|u - u_K\|\leq \fa^{-1}\|r_K\|_{\cH'},
\]
which trivially holds for $u = u_K$ as well.
Hence, we derive an upper bound for $\|r_K\|_{\cH'}$. 
}

\MOD{Let $G\in\cG^{(K)}$ and $\tilde G\in\cG^{(L)}$ for some $L>K$ such that $\tilde G\subset G$. Furthermore, let $\nu$ and $\tilde \nu$  be the outer normal of $G$ and $\tilde G$ respectively. We first observe that $-\Delta u_K=f$ on $G$ with $-\partial_\nu u_K=\jump{u_K}$ on $\partial G$. In particular, 
we note that $\nabla u_K\cdot\nu\in L^2(\partial G)$. Furthermore, by the regularity obtained in Lemma~\ref{lem:REGULARITY}, we see that $\nabla u_K\cdot\tilde\nu\in L^2(\partial \tilde G)$. 
Now we can use a version of Green's formula proved by Casas and Fern{\'a}ndez \cite[Corollary 1]{casas1989green}, 
exploiting (in the notation of \cite{casas1989green}) 
that $\nabla u_K\in W^2(\text{div},G)$ and $v\in W^1(\tilde G)\cap L^\infty(\tilde G)$,
to obtain 
\begin{equation} \label{eq:RESREP}
r_K(v)=\ell(v)-a(u_K,v) = \sum_{k=K+1}^{L} \int_{\Gamma_k}\partial_{\nu}u_K\jump{v} \; d\Gamma_k
\end{equation}
for any test function $v\in\cC^1_{L,0}(\bQ)$.
The Cauchy-Schwarz inequality then yields
\[
\begin{array}{rl}
r_K(v) & = \displaystyle \sum_{k=K+1}^L \int_{\Gamma_k} \left( (1+\fc)^{-k/2}C_k^{-1/2}\partial_{\nu}u_K\right)
\left(  (1+\fc)^{k/2}C_k^{1/2}\jump{v}\right) \; d\Gamma_k \\
& \leq \displaystyle \left(\sum_{k=K+1}^L \int_{\Gamma_k}  (1+\fc)^{-k}C_k^{-1}|\partial_{\nu}u_K|^2\; d\Gamma_k \right)^{1/2}\|v\| \; .
\end{array}
\]
Since $L$ can be arbitrarily large, we infer
}
\[
\|r_K\|_{\cH'}^2\leq  \left(\sup_{k>K}C_k^{-1}d_k^{-1}\right)
\left(\sup_{k>K}d_k \|\partial_{\nu} u_K\|^2_{L^2{(\Gamma_k})}\right) (1+\fc)^{-K}
\left(\sum_{k=1}^\infty  (1+\fc)^{-k} \right)
\]
and Lemma~\ref{lem:REGULARITY} provides the a priori estimate
\[
\begin{array}{rl}
 \displaystyle d_k \|\partial_{\nu}u_K\|_{L^2(\Gamma_k)}^2
&= \displaystyle \sum_{G \in \cG^{(K)}\setminus \cG^{(K)}_\infty}
\sum_{i=1}^d \displaystyle d_k \|\partial_i u_K\|_{L^2(\gamma_{k,D,i})}^2  \\
&\displaystyle\leq  3d c \left( d_k^2\|f\|_{L^2(\bQ)}^2 + \|\nabla u_K\|^2_{L^2(\bQ\setminus \Gamma^{(K)})}\right)
\leq  C\|f\|_{L^2(\bQ)}^2
\end{array}
\]
for all $k>K$ with $C$ depending only on $d$, $\fg$, $\fa$, 
and the Poincar\'e-type constant $C_0$ in \eqref{eq:Poincare-H-1-2}. 
This concludes the proof.
\end{proof}
Recall that the factor  $\sup_{k > K}C_k^{-1}d_k^{-1}$ depends on the geometry of the actual interface network. 
\begin{rem} \label{rem:CANTORERROR}
For the Cantor interface network described in Example~\ref{exa:Cantor},
we have  $C_K^{-1} d_K^{-1} = 2$ for all $ K\in \N$.
Hence, Theorem~\ref{thm:ERREST} implies exponential convergence 
of  the solution $u_K$ of the level-$K$ interface Problem~\ref{prob:MultiscaleMin}
to the solution $u$ of the fractal interface Problem~\ref{prob:FractalMin}
according to the error estimate
\[
\|u-u_K\| \leq C \|f\|_{L^2(\bQ)} (1 + \fc)^{-K}
\]
with $C$ only depending on $d$, $\fg$, $\fa$, \MOD{$\fc$, and
on the Poincar\'e-type constant $C_0$ in \eqref{eq:Poincare-H-1-2}.} 
\end{rem}
\MOD{
\begin{rem} \label{rem:EXPONENTIALGROWTH}
The exponential decay of $\|u-u_K\| $ is essentially 
due to the exponential growth of the weights $(1+ \fc)^k$ on the interfaces. 
It is an interesting question for future investigations whether these weights 
can be replaced by another monotonically increasing function $f(k)$. 
However, note that the Poincar\'e-type inequality \eqref{eq:Poincare-H-1-2}  
indicates exponential growth of $f(k)$.
\end{rem}
}

\section{Numerical computations}\label{sec:numerics}

Let $\cT^{(1)}$ be a partition of $\bQ$ 
into simplices with maximal diameter $h_{1} > 0$
which is regular in the sense that the intersection of two simplices from
$\cT^{(1)}$ is either a common $n$-simplex for some $n=0,\dots,d$ or empty.
Then $\cT^{(K)}$ denotes the partition of $\bQ$  resulting from $K-1$ uniform regular 
refinements of $\cT^{(1)}$ (cf., e.g., \cite{bank1983some,bey2000simplicial}) for each $K\in \N$.
The maximal diameter is $h_{K} = h_1 2^{K-1}$,
and $\cN^{(K)}$ stands for the set of vertices of simplices in $\cT^{(K)}$.
We assume that the partition 
$\cT^{(K)}$ resolves the piecewise affine interface network $\Gamma^{(K)}$,
i.e., for all $k\leq K$ the interfaces $\Gamma_k$ can be represented as a sequence of 
$(d-1)$-faces of simplices from  $\cT^{(K)}$.
For each $K\in \N$ and each $G\in \cG^{(K)}$, 
we introduce the space $\cS_G^{(K)}$ of piecewise affine finite elements 
with respect to the local partition $\cT_G^{(K)}=\{T\in \cT^{(K)}\;|\; T\subset \overline{G}\}$.
The discretization of the level-$K$ interface Problem~\ref{prob:MultiscaleMin}
with respect to the corresponding broken finite element space
\[
\cS^{(K)} = \{v: \bQ \to \R\;|\quad v|_G \in S^{(K)}_G \; \forall G\in \cG^{(K)} \}\subset \cH_K
\]
amounts to finding $\tilde{u}_K \in \cS^{(K)}$ such that 
\begin{equation} \label{eq:FED}
 a(\tilde{u}_K,v)=\ell(v)\qquad \forall v\in \cS^{(K)}\; .
\end{equation}
For each $K\in \N$, existence and uniqueness of a solution 
follows from the Lax-Milgram lemma.

\subsection{Exponential convergence of multiscale interface problems} \label{subsec:ExponentialConvergence}
In case of the Cantor interface network (cf.\ Example~\ref{exa:Cantor})
the solutions $u_K$ of the level-$K$ interface Problem~\ref{prob:MultiscaleMin} for $K\in \N$ 
converge exponentially to the solution $u$ of the 
fractal interface Problem~\ref{prob:FractalMin}
(cf.\ Remark~\ref{rem:CANTORERROR}).
For a numerical illustration,
we consider this example in $d=2$ space dimensions
with $\fc = 1$, $f \equiv 1$, $A\equiv 1$, and the geometrical parameter $C_K=2^{K-1}$.
Note that the $\|\cdot\|$ norm in $\cH$ (cf.\ \eqref{eq:scrH-inner-product}) is 
identical with the energy norm induced by $a(\cdot,\cdot)$ (cf.\ \eqref{eq:BLF}) in this instance.

\begin{figure}[h]
	\centering

	\includegraphics[width=0.3\linewidth]{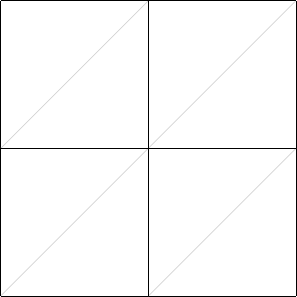}\quad
	\includegraphics[width=0.3\linewidth]{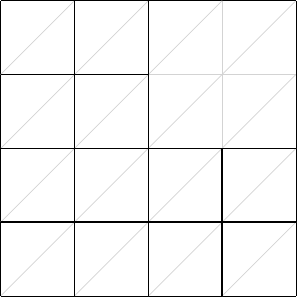}\quad
	\includegraphics[width=0.3\linewidth]{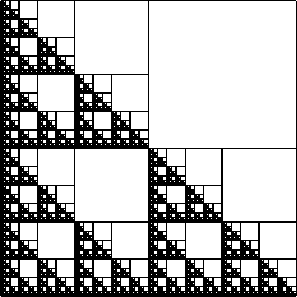}
	\caption{Initial triangulation $\cT^{(1)}$, uniform refinement $\cT^{(2)}$ together with the Cantor interface network $\Gamma^{(K)}$ for $K=1$, $2$, and $8$ in $d=2$ space dimensions} \label{fig:2d-levelinterface-network}
\end{figure}

The initial triangulation $\cT^{(1)}$ with $h_1=2^{-1}$ 
is depicted in the left picture of Figure~\ref{fig:2d-levelinterface-network} (grey)
together with the initial Cantor network $\Gamma^{(1)}$ (black). 
Successive uniform refinement of $\cT^{(1)}$ provides 
the triangulations $\cT^{(K)}$ with $h_K=2^{-K}$ 
resolving the interfaces $\Gamma^{(K)}$ on subsequent levels $K$. 
The case $K=2$ is illustrated in the middle 
while the right picture of Figure~\ref{fig:2d-levelinterface-network} shows 
the Cantor network $\Gamma^{(8)}$.

The linear systems associated with the corresponding finite element discretizations \eqref{eq:FED}
on each level $K$ are solved directly.
Exploiting 
\[
\Vert u - u_K \Vert \leq \Vert u - u_{9} \Vert +   \Vert u_{9} - u_K \Vert\; ,\qquad K \in \N\; ,
\]
the fractal homogenization error is replaced by the heuristic error estimate
\begin{equation} \label{eq:HEUERR}
e_K=\Vert \tilde{u}_{10} - \tilde{u}_{9} \Vert +   \Vert \tilde{u}_{9} - \tilde{u}_K \Vert\; ,
\qquad K = 1,\dots, 8\; .
\end{equation}
\MOD{The first term 
in \eqref{eq:HEUERR} is intended to capture the error  
made by resolving a ``large'' but finite number  of interfaces
instead of infinitely many, while the second term 
aims at the additional contribution made by resolving 
only the actual ``small''  number of $K=1,\dots, 8$ levels.}

Figure~\ref{fig:convergence-test} 
shows the error estimates $e_K$ over the levels $K$ (dotted line)
together with the expected asymptotic bound of order $(1+\fc)^{-K}$ (solid line)
for $K=1,\dots, 8$. 
Both curves have very similar slope which nicely confirms our theoretical findings.
\MOD{
As $\|u - u_9\|\geq \|u_{10} - u_9\|$ and $\|\tilde{u}_9-\tilde{u}_K\|=0$ for $K=9$, 
we would expect that $e_K$  underestimates the fractal homogenization error 
for increasing $K$. 
This could explain the slight deviation from the expected asymptotic behavior. 
}

\begin{figure}
	\centering
	\includegraphics[width=0.4\linewidth]{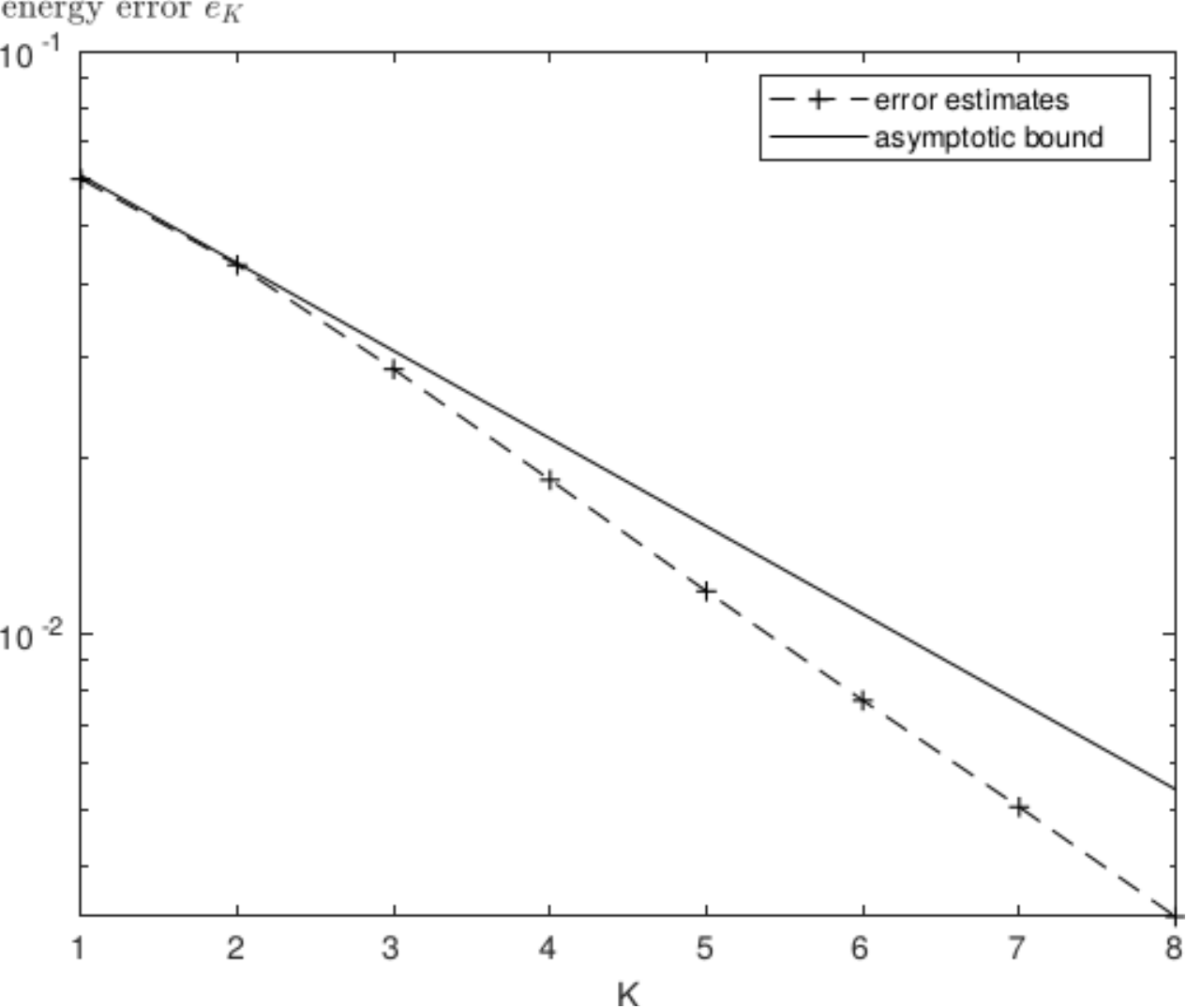}
	\caption{Exponential decay of fractal homogenization error} \label{fig:convergence-test}
\end{figure}

\subsection{\MOD{Fractal} numerical homogenization}\label{sec:numhom}
Aiming at an iterative solution of the discrete problems \eqref{eq:FED}
with a convergence speed that is independent of the number of levels $K\in \N$,
we now present a multilevel preconditioner 
in the spirit of~\cite{KornhuberPodlesnyYserentant17,KornhuberYserentant16}.

To this end, we introduce the sets of local patches 
\begin{equation*}
\bQ^{(k)} = \left\{
\begin{array}{ccl}
\lbrace \overline{\bQ} \rbrace &\text{ for }& k=1\\
\left\lbrace \bomega_{x}^{(k)} \subset \overline{\bQ} \,\vert \; x \in \cN^{(k-1)}\right\rbrace & \text{ for } & k\geq 2
\end{array}
\right .
\end{equation*}
with $\bomega_{x}^{(k)} \subset \overline{\bQ}$ 
consisting of all simplices $T\in \cT^{(k-1)}$ with common vertex $x\in \cN^{(k-1)}$. 
The decomposition of $\bQ$ into patches $\bomega \in \bQ^{(k)}$ gives rise to the decomposition
\begin{equation}
 S^{(k)}=\sum_{\bomega \in \bQ^{(k)}} S_{\bomega}^{(k)}\; ,\qquad k\in \N\; ,
\end{equation}
into the local finite element spaces
\[
S_{\bomega}^{(k)} =\{v\in S^{(k)} \;\vert \; v|_{\bQ \backslash \text{int } \bomega}=0 \}\;,
\qquad  \omega \in \bQ^{(k)}\; .
\]
For each fixed $K\in \N$, this leads to the splitting 
\[
S^{(K)} = \sum_{k=1}^K \sum_{\bomega \in \bQ^{(k)}} S_{\bomega}^{(k)}
\]
and the corresponding multilevel preconditioner~\cite{XuSubspaceCorrection,Yserentant93}
\begin{equation} \label{eq:discrete-preconditioner}
T_{K} = \sum\limits_{k=0}^{K} \sum\limits_{\bomega \in \bQ^{(k)}} P_{S_{\bomega}^{(k)}}.
\end{equation}
with $P_{V} : S^{(K)} \longrightarrow V$ denoting the Ritz projection,

defined by
\begin{equation} \label{eq:projections}
a( P_{V} w, v ) = a( w, v ), \quad \forall v \in V.
\end{equation}
Note that the evaluation of each local projection $P_{S_{\bomega}^{(k)}}$ 
amounts to the solution of a (small)
self-adjoint linear system on the  patch $\bomega\in \bQ^{(k)}$. 
Therefore, $T_{K}$ can be regarded as a multilevel version 
of the classical block Jacobi preconditioner.

The analysis of upper bounds for the condition number of $T_{K}$ as well as 
fractal counterparts of multiscale finite elements~\cite{KornhuberPeterseimYserentant16,MalqvistPeterseim14}
will be considered in a separate publication.

\subsubsection{Cantor interface network}\label{subsubsec:Cantor}
We consider the level-$K$ interface Problem~\ref{prob:MultiscaleMin}
for the Cantor interface network
with parameters, finite element discretization, and initial triangulation $\cT^{(1)}$
as previously described in  subsection \ref{subsec:ExponentialConvergence}.

Let $\tilde{u}_K^{(\nu)}$, $\nu \in \N$, denote the iterates of 
the preconditioned conjugate gradient method with preconditioner $T_{K}$ given in \eqref{eq:discrete-preconditioner} and initial iterate $\tilde{u}_K^{(0)} = \tilde{u}_0$.
The corresponding algebraic error reduction factors
\begin{equation} \label{eq:ERRRED}
\rho_K^{(\nu)}=\frac{\|\tilde{u}_K - \tilde{u}_K^{(\nu)}\|}{\|\tilde{u}_K - \tilde{u}_K^{(\nu-1)}\|}
\end{equation}
of each iteration step are depicted in Figure~\ref{fig:CNVRATESCANTOR} for $\nu=1,\dots,8$ 
together with their geometric average $\rho_K$ for $K=5,\dots, 9$.
The averaged reduction factors $\rho_K$ seem to saturate with increasing level $K$.

\begin{figure}[h] 
\centering
\begin{tabular}{cccccc}
	$\nu$ & $K=5$ & $K=6$ & $K=7$ & $K=8$ & $K=9$\\
	\hline
	$1$ & $0.479$ & $0.481$ & $0.481$ & $0.482$ & $0.482$\\
	$2$ & $0.445$ & $0.464$ & $0.483$ & $0.500$ & $0.514$\\
	$3$ & $0.453$ & $0.448$ & $0.442$ & $0.437$ & $0.439$\\
	$4$ & $0.429$ & $0.452$ & $0.474$ & $0.493$ & $0.503$\\
	$5$ & $0.451$ & $0.465$ & $0.468$ & $0.472$ & $0.477$\\
	$6$ & $0.432$ & $0.444$ & $0.459$ & $0.477$ & $0.494$\\
	$7$ & $0.447$ & $0.467$ & $0.463$ & $0.456$ & $0.455$\\
	$8$ & $0.450$ & $0.483$ & $0.487$ & $0.489$ & $0.490$\\
	\hline
	$\rho_K$ & $0.448$ & $0.463$ & $0.469$ & $0.475$ & $0.481$\\
	\hline
\end{tabular}
\caption{Algebraic error reduction factors for the Cantor interface network} \label{fig:CNVRATESCANTOR}
\end{figure}

In practical computations, 
it is sufficient to reduce the algebraic error $\|\tilde{u}_K - \tilde{u}_K^{(\nu)}\|$
up to discretization accuracy $\|u_K - \tilde{u}_K\|$. Galerkin orthogonality implies
\[
  \|\tilde{u}_{K+1} - \tilde{u}_K\|^2 + \|u - \tilde{u}_{K+1}\|^2 = \|u- \tilde{u}_K\|^2\; .
\]
We utilize the stopping criterion 
\begin{equation} \label{eq:STOP}
\|\tilde{u}_K - \tilde{u}_K^{(\nu_0)}\| \leq \|\tilde{u}_{K+1} - \tilde{u}_K\|\leq \|u- \tilde{u}_K\|
\end{equation}
provided by the resulting lower bound for the discretization error
and the final iterate on the preceding level $K-1$ 
as the initial iterate on the actual level $K$ (nested iteration).
Then, only $\nu_0 = 1$ step of the preconditioned \MOD{conjugate gradient} iteration
is sufficient to provide an approximation $\tilde{u}_K^{(\nu_0)}$ of $\tilde{u}_K$ 
with discretization accuracy for all $K=2,\dots, 9$.

\subsubsection{Layered interfaces} 
We consider the level-$K$ interface Problem~\ref{prob:MultiscaleMin}
in $d=2$ space dimensions
with parameters $\fc = 1$, $f \equiv 1$, $A\equiv 1$, 
and non-intersecting interfaces $\Gamma_k \subset \bQ=(0,1)^2$ described as follows.
Figure~\ref{fig:geointerfaces} shows the initial triangulation $\cT^{(1)}$ (grey) with $h_1=2^{-4}$
together with the 3 macro interfaces forming $\Gamma^{(1)}$.
Again, $\cT^{(k)}$ is obtained by uniform refinement of $\cT^{(1)}$
and $\Gamma_k=\Gamma^{(K)} \backslash \Gamma^{(K-1)}$ is composed of 
6 randomly selected, non-intersecting polygons consisting of edges of triangles $T\in \cT^{(K)}$ 
one above and one below each macro interface from $\Gamma^{(1)}$.
For $K=2$, this is illustrated in the middle picture of Figure~\ref{fig:geointerfaces}.
Note that at most $C_K = 2^K - 1$ interfaces are cut by any straight line through $\bQ$.
The final interface $\Gamma^{(6)}$ is displayed in the right picture. 
\begin{figure}[h]
	\centering
	\includegraphics[width=.3\linewidth]{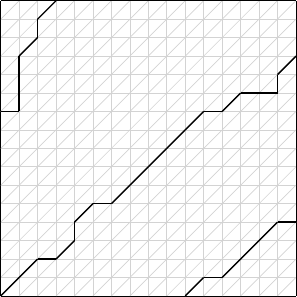}\quad
	\includegraphics[width=.3\linewidth]{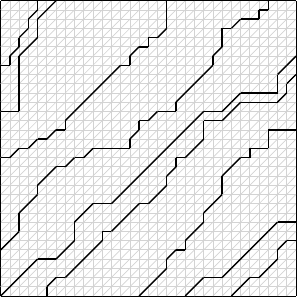}\quad
	\includegraphics[width=.3\linewidth]{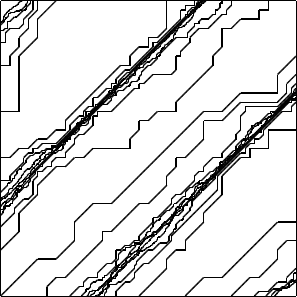}
	\caption{Initial triangulation $\cT^{(1)}$, uniform refinement $\cT^{(2)}$ together with the layered interface network $\Gamma^{(K)}$ for $K=1$, $2$, and $6$} \label{fig:geointerfaces}
\end{figure}

\begin{figure}[h]
\begin{tabular}{cccccc}
	$\nu$ & $K=2$ & $K=3$ & $K=4$ & $K=5$ & $K=6$\\
	\hline
	$1$ & $0.303$ & $0.377$ & $0.376$ & $0.392$ & $0.423$ \\
	$2$ & $0.177$ & $0.431$ & $0.473$ & $0.496$ & $0.523$ \\
	$3$ & $0.329$ & $0.316$ & $0.418$ & $0.492$ & $0.542$ \\
	$4$ & $0.372$ & $0.404$ & $0.405$ & $0.497$ & $0.517$ \\
	$5$ & $0.247$ & $0.409$ & $0.501$ & $0.503$ & $0.525$ \\
	$6$ & $0.329$ & $0.405$ & $0.421$ & $0.497$ & $0.533$ \\
	$7$ & $0.366$ & $0.362$ & $0.458$ & $0.48$8 & $0.539$ \\
	$8$ & $0.328$ & $0.440$ & $0.426$ & $0.497$ & $0.527$ \\
	\hline
	$\rho_K$ & $0.299$ & $0.391$ & $0.433$ & $0.481$ & $0.515$ \\
	\hline
\end{tabular}	
\caption{Algebraic error reduction factors for the layered interface network} \label{tab:geo-reduction-factors}
\end{figure}

As in Subsection~\ref{subsubsec:Cantor}, 
we consider the conjugate gradient iteration with the multilevel preconditioner 
defined in \eqref{eq:discrete-preconditioner} 
and  initial iterate $u_K^{(0)} = \tilde{u}_0$ for $K=2,\dots,6$.
Figure~\ref{tab:geo-reduction-factors} 
shows the algebraic error reduction factors $\rho_K^{(\nu)}$ defined in \eqref{eq:ERRRED},
together with their geometric average $\rho_K$.
The averaged reduction factors $\rho_K$ are slightly increasing with increasing level $K$.

If nested iteration is applied, only one iteration step is needed to 
reach discretization accuracy according to the stopping criterion \eqref{eq:STOP}
\MOD{in Subsection~\ref{subsubsec:Cantor}}.

\bibliographystyle{plain}
\bibliography{HeidaKornhuberPodlesny2018}

\end{document}